\documentclass[12pt,reqno]{amsart}
\usepackage{amssymb}
\usepackage{a4wide}

\usepackage{enumerate}

\newtheorem{theorem}{Theorem}[section]
\newtheorem{lemma}[theorem]{Lemma}
\newtheorem{corollary}[theorem]{Corollary}

\theoremstyle{definition}
\newtheorem{definition}[theorem]{Definition}

\theoremstyle{remark}

\newcommand{\itemprefix}{}
\makeatletter
\newcommand{\myitem}{%
\item\protected@edef\@currentlabel{\itemprefix\theenumi}%
}
\makeatother

\def\arhang{Arhangel'skii}

\def\con{\subseteq}

\def\reals{\mathbb{R}}
\def\naturals{\mathbb{N}}
\def\interval{\mathbb{I}}
\def\from{\colon}
\def\cont{\mathfrak{c}}
\newcommand\ZFC{\ensuremath{\mathsf{ZFC}}}
\def\weight#1{w(#1)}
\def\Cech{\v{C}ech}
\def\CS{\Cech-Stone}
\def\orpr#1#2{\langle #1,#2 \rangle}

\makeatletter
\def\int{\mathop{\operator@font int}\nolimits}
\makeatother

\author[I. Juh\'asz]{Istv\'an Juh\'asz}
\address      { Alfr\'ed Rényi Institute of Mathematics%
}
\email{juhasz@renyi.hu}

\author[J. van Mill]{Jan van Mill}
\address{University of Amsterdam}
\email{j.vanMill@uva.nl}

\thanks{The first author was supported
by NKFIH grant no. K129211.}

\makeatletter
  \@namedef{subjclassname@2020}{%
  \textup{2020} Mathematics Subject Classification}
   \makeatother

\subjclass[2020]{54A25, 54C05, 54D30}
\keywords{Homogeneous space, topological group, weight, continuous image.}

\title{Homogeneous continuous images of smaller weight}

\begin{document}

\begin{abstract}
We show that every infinite crowded space can be mapped onto a homogeneous space of countable weight, and that there is a homogeneous space of weight $\cont$ that cannot be mapped onto a homogeneous space of weight strictly between $\omega$ and $\cont$.
\end{abstract}

\maketitle

\section{Introduction}
For a space $X$, we let $w(X)$ denote its weight.
Tkachenko~\cite{Tkachenko80} proved that if $X$ is a Tychonoff space, then for every infinite regular cardinal $\tau \le w(X)$,
there is a continuous image of $X$ of weight $\tau$.
The same result was obtained independently but later by Juh\'asz in~\cite{Juhasz96}. Several related `reflecting' properties were found for the class of $\omega$-narrow topological groups in~\cite{HernandezTkachenko17}. It was shown there for example that every $\omega$-narrow topological group $G$ admits a continuous homomorphism onto a topological group of weight $\tau$, where $\tau\le w(G)$ is any infinite regular cardinal. See \cite{HernandezTkachenko17} for more results and references. It is an intriguing problem whether the assumption about the regularity of $\tau$ is essential in these results.

Although it is not stated explicitly in~\cite{HernandezTkachenko17}, their assumption on $\omega$-narrowness is essential. Examples are easily found. Koppelberg~\cite{koppelberg:homogeneousba} proved that the homeomorphism group of any $D(2)^\kappa$, $\kappa\ge \omega$, is algebraically simple (for $\kappa=\omega$, this is due to Anderson~\cite{Anderson58a}). Hence the homeomorphism group of $D(2)^{\omega_1}$ endowed with the compact-open topology, being of weight $\omega_1$, does not admit a continuous homomorphic image of weight $\omega$.

As far as we know, for the class of topologically homogeneous spaces no `reflecting' properties were obtained in the literature. Since there are many for topological groups, it is a natural question whether these can be extended to the more general class of homogeneous spaces. See \cite{ArhangMill2014} for a survey containing old and new problems on homogeneous spaces. We show that the results on topological groups mentioned above cannot be extended in \ZFC\ to homogeneous spaces by proving that every infinite crowded space can be mapped onto an infinite homogeneous space of countable weight, and that there is a homogeneous space of weight $\cont$ that cannot be mapped onto any homogeneous space of weight strictly between $\omega$ and $\cont$. We also comment on the cardinality of (generalizations of) almost-compact spaces.

\section{Continuous images of countable weight}
In this and the next section, all spaces under discussion are Tychonoff.

As usual, $\naturals$, $\reals$ and $\interval$ denote the spaces of natural numbers, real numbers and the closed interval $[0,1]$, respectively.

\begin{lemma}
Every infinite crowded space maps onto $\naturals$ or $\interval$.
\end{lemma}

\begin{proof}
Assume that there is a continuous map $f\from X\to \reals$ such that $f(X)$ is not zero-dimensional. Then $f(X)$ contains a nontrivial closed interval $J$ on which $f(X)$ can be retracted.
Hence we may assume that for all continuous functions $f\from X\to \reals$, $f(X)$ is zero-dimensional. As a consequence, $X$ is zero-dimensional.

Assume that there is a continuous function $f\from X\to \reals$ such that $f(X)$ is not compact. Then $f(X)$ contains a closed copy of $\naturals$ on which $f(X)$, being zero-dimensional, can be retracted.
Hence we can assume that for all continuous functions $f\from X\to\reals$, $f(X)$ is compact; that is, $X$ is pseudocompact.

Assume that there is a continuous function $f\from X\to \reals$ such that $f(X)$ is uncountable. Then $f(X)$, being compact, contains a Cantor set $K$. And since $f(X)$ is zero-dimensional, it can be retracted onto $K$. Now it suffices to observe that $K$ can be mapped onto $\interval$.

Let $bX$ be a zero-dimensional compactification of $X$. Since $b X$ is crowded and zero-dimensional, there is a continuous surjection $f\from b X\to D(2)^\omega$. But then by pseudocompactness of $X$, $f(X) = D(2)^\omega$, and so we are done by what we just observed.
\end{proof}

\begin{corollary}
Every infinite crowded space can be mapped onto an infinite homogeneous space of countable weight.
\end{corollary}

\begin{proof}
Simply observe that by collapsing the two endpoints of $\interval$ to a single point, we obtain a homogeneous space.
\end{proof}

The lemma and corollary are not true if $X$ is not crowded. For let $X=W(\omega_1)$, the space of all countable ordinal numbers. Then every continuous image of $X$ of countable weight, is compact and countable. Hence has an isolated point. Hence if it is infinite, it is neither $\naturals$ nor $\interval$, and not homogeneous.

\section{Continuous images of uncountable weight}\label{derdeparagraaf}

We now formulate and prove the main result in this note.

\begin{theorem}\label{nulde}
There is a homogeneous space of weight $\cont$ such that if $f\from X\to Y$ is a continuous surjection and $\omega < \weight{Y} < \cont$, then $Y$ is not homogeneous.
\end{theorem}

Let $\mathcal{A}$ be a Mr\'owka family on $\omega$. That is, $\mathcal{A}$ is a MAD-family such that the \CS\ compactification of $\Psi(\mathcal{A})$, the $\Psi$-space $\omega\cup \mathcal{A}$ of $\mathcal{A}$, coincides with its 1-point compactification. That is, $\Psi(\mathcal{A})$ is \emph{almost-compact}. That such a family exists is well-known, Mr\'owka~\cite{Mrowka77} (see also~\cite[8.6.1]{HernandezHernandezHrusak18}). The family $\mathcal{A}$ has cardinality $\cont$ by construction, hence $\weight{\Psi(\mathcal{A})} = \cont$.

Now in $\Psi(\mathcal{A})$ we replace every point of $\omega$ by a copy of the Cantor set $K=D(2)^\omega$ to obtain a space $X$, as follows. The underlying set of $X$ is $(\omega\times K) \cup \mathcal{A}$. A basic neighborhood of a point $\orpr{n}{k}$, where $n\in \omega$ and $k\in K$, has the form $\{n\}\times C$, where $C$ is any open neighborhood of $k$ in $K$. And a basic neighborhood of $A\in \mathcal{A}$ in $X$ has for $n \in \omega$ the form $U_n(A)=\{A\}\cup \bigcup \{\{m\} \times K : m\in A, m\ge n\}$. It is clear that $X$ is zero-dimensional and locally homeomorphic to $K$, hence $X$ is homogeneous.

Let $f\from X\to \Psi(\mathcal{A})$ be the function that sends $A$ in $X$ to $A$ in $\Psi(\mathcal{A})$, for every $A\in \mathcal{A}$, and every $\{n\}\times K$ to $\{n\}$, for $n\in \omega$.

\begin{lemma}\label{eerste}
$f$ is perfect and open.
\end{lemma}

\begin{proof}
It is clear that $f$ is continuous and has compact fibers. It is also clear that $f$ is open.
Hence it suffices to prove that $f$ is closed. To this end, let $E$ be any closed subset of $X$. We claim that that $\Psi(A)\setminus f(E)$ is open. If $p\in  (\Psi(\mathcal{A})\setminus f(E)) \cap \omega$, then $\{p\}$ is a neighborhood of $p$ that is contained in $\Psi(\mathcal{A})\setminus f(E)$. Now assume that $p = A\in \mathcal{A}$. Then $A$ in $X$ does not belong to $E$. Hence there exists $n$ such that $U_n(A) \cap E=\emptyset$. But $f^{-1}(f(U_n(A)) = U_n(A)$, hence $f(U_n(A)) \cap f(E) = \emptyset$. But $f(U_n(A))$ is a neighborhood of $A$ in $\Psi(\mathcal{A})$, hence we are done.
\end{proof}

\begin{lemma}\label{tweede}
If $Z$ is a zero-set $X$, then $f(Z)$ is a zero-set in $\Psi(\mathcal{A})$.
\end{lemma}

\begin{proof}
Let $\xi\from X\to [0,1]$ be continuous such that $\xi^{-1}(0) = Z$. Define $\eta\from \Psi(\mathcal{A}) \to [0,1]$ as follows:
$\eta(p) = \min \xi( f^{-1}(p))$. It is clear that $\eta$ is well-defined since $f$ is perfect (Lemma~\ref{eerste}). To prove it is continuous, we only need to check that at some $A\in \mathcal{A}$. By first-countability, we can check that by considering convergent sequences. A typical sequence that converges to $A$ in $\Psi(\mathcal{A})$ has the form $C=\{m\in A : m \ge n\}$ for some $n$. By the definition of the topology on $X$, $\{f^{-1}(c) : c\in C\}$ converges to $A$ in $X$. Hence $\{\xi(f^{-1}(c)) : c\in C\}$ converges to $\eta(A) = \xi(A)$ in $[0,1]$. But then so do the minima of these compact sets.

We will shows that $\eta^{-1}(0) = f(Z)$, which does the job. Indeed, pick an arbitrary $p\in \eta^{-1}(0)$. Assume first that $p = A\in \mathcal{A}$. Then since $\eta(A) = \xi(A)$, $A\in Z$, hence $A=f(A)\in f(Z)$. Assume next that $ p = n\in \omega$. Then $f^{-1}(p) = \{n\}\times K$, hence there exists $k\in K$ such that $\xi(\orpr{n}{k})=0$. Hence $\orpr{n}{k}\in Z$, so that $n = f(\orpr{n}{k}) \in f(Z)$. From this we conclude that $\eta^{-1}(0) \con f(Z)$. For the reverse inclusion, take an arbitrary $z\in Z$, and consider $f(z)$. Again, there are two cases. Assume first that $f(z) = A\in \mathcal{A}$. Then $z= A$, hence $\eta(f(z)) = \xi(A)=0$. Assume next that $f(z) = n\in \omega$. Then $z\in \{n\}\times K$, hence $\min (\xi(\{n\}\times K)) = 0$. That is, $\eta(f(z))=\eta(n) = 0$.
\end{proof}

\begin{corollary}\label{derde}
$X$ is almost compact.
\end{corollary}

\begin{proof}
By Gillman and Jerison~\cite[6J]{gj} we must show that of any two disjoint zero-sets in $X$, at least one of them is compact. Hence, assume that $Z_0$ and $Z_1$ are disjoint zero-sets in $X$ such that neither $Z_0$ nor $Z_1$ is compact.

Assume that $Z_0\cap \mathcal{A}$ is finite. There is a clopen compact $C$ in $X$ that contains $Z_0\cap \mathcal{A}$. Hence $S_0 = Z_0\setminus C$ is a noncompact zero-set in $X$ that misses $\mathcal{A}$. But then $f(S_0)$ is by Lemma~\ref{eerste} an infinite closed subset of $\Psi(\mathcal{A})$ that misses $\mathcal{A}$. But every infinite closed subset of $\Psi(\mathcal{A})$ intersects $\mathcal{A}$ since $\mathcal{A}$ is MAD.

We may consequently assume without loss of generality that $Z_0\cap \mathcal{A}$ and $Z_1\cap \mathcal{A}$ are both infinite. But they are clearly zero-sets of $X$. But then by Lemma~\ref{tweede}, $\mathcal{A}$ contains two disjoint infinite zero-sets of $\Psi(\mathcal{A})$, which contradicts $\Psi(\mathcal{A})$ being almost compact.
\end{proof}

Write $\beta \Psi(\mathcal{A}) = \Psi(\mathcal{A}) \cup \{\infty\}$.

We are now in a position to present the proof of Theorem~\ref{nulde}. To this end, let $g\from X\to Y$ be continuous, where $\omega < \weight{Y} < \cont$ and $Y$ is homogeneous. We will show that this leads to a contradiction.

The space $Y$ is almost compact (hence locally compact) by Gillman and Jerison~\cite[6J]{gj}.

Write $\beta X = X\cup \{\infty_X\}$, $\beta Y = Y \cup \{\infty_Y\}$, and let $h=\beta g\from \beta X\to \beta Y$ be the Stone extension of $g$.

\begin{lemma}\label{vierde}
$Y$ is compact.
\end{lemma}

\begin{proof}
Assume that $Y$ is not compact. Then $h(\infty_X) = \infty_Y$, and since $h(X)=Y$, we get $h^{-1}(\{\infty_Y\}) = \{\infty_X\}$. Since $Y$ is locally compact, $\weight{\beta Y} < \cont$ from which it follows that the character of $\infty_X$ in $X$ is less than $\cont$. But then the character of $\infty$ in $\beta\Psi(\mathcal{A})$ is less than $\cont$, which is a contradiction.
\end{proof}

Put $y_0 = h(\infty_X)$. Take an arbitrary $y_1\in X\setminus \{y_0\}$. Let $U$ be a compact neighborhood of $y_1$ in $Y$ that misses $y_0$. Then $h^{-1}(U)$ is a compact subset of $X$ and hence is of countable weight. But then $U$ is of countable weight. By homogeneity, every point of $Y$ has a compact neighborhood of countable weight. But then by compactness, $Y$ has countable weight, which is a contradiction.

\section{On locally $\mathfrak{c}$ and $\mathfrak{c}$-fair spaces}

In the previous section, we constructed an almost-compact first-countable and homogeneous space of weight and cardinality $\cont$. By \arhang's celebrated
result from~\cite{arh:c1bicompacta}, every first-countable Lindel\"of space has cardinality at most $\cont$. In the light of this it is natural
to wonder whether there is a bound on the cardinality of first-countable almost-compact spaces. Our following general results answer this, but we
think they are of independent interest in themselves.

For the spaces in this section, unless otherwise stated, no separation axiom is assumed.

\begin{definition}
A topological space $X$ is
\begin{enumerate}[(i)]
  \item {\em locally $\mathfrak{c}$} if every point in $X$ has a neighborhood of cardinality $\le \,\mathfrak{c}$;

  \item {\em $\mathfrak{c}$-fair} if the closure of every subset of $X$ of cardinality $\mathfrak{c}$ also has cardinality $\mathfrak{c}$.

  \item We call a $\kappa$-sequence $\{F_\alpha : \alpha < \kappa\}$ of subsets of $X$ {\em strongly increasing}
  if $F_\alpha \subset\int F_{\alpha+1}$ for all $\alpha < \kappa$.
\end{enumerate}
\end{definition}

We now present the main result of this section.

\begin{theorem}\label{tm:lcf}
(i) Assume that the space $X$ is both locally $\mathfrak{c}$ and $\mathfrak{c}$-fair,
moreover $\kappa \le \mathfrak{c}$ is a regular cardinal such that for every strongly increasing
$\kappa$-sequence $\{F_\alpha : \alpha < \kappa\}$ of closed subsets of $X$ their union
$\bigcup \{F_\alpha : \alpha < \kappa\}$ is closed in $X$. Then $|X| > \mathfrak{c}$ implies
that for every subset $A$ of $X$ with $|A| \le \mathfrak{c}$
there exists a {\em clopen} subset $U$ of $X$ with $A \subset U$ and $|U| = \mathfrak{c}$ such that
$L(U) \ge \kappa$.

\smallskip

(ii) If  $X$ is both locally $\mathfrak{c}$ and $\mathfrak{c}$-fair,
moreover we have $t(x, X) < cf(\mathfrak{c})$ for all points $x \in X$, then $|X| > \mathfrak{c}$ implies the existence of
a clopen subset $U$ of $X$ with $|U| = L(U) = \mathfrak{c}$.
\end{theorem}

\begin{proof}
(i) It is straightforward from $|X| > \mathfrak{c}$ and $X$ being locally $\mathfrak{c}$ and $\mathfrak{c}$-fair, that we may define
by transfinite recursion a strongly increasing
$\kappa$-sequence $\{F_\alpha : \alpha < \kappa\}$ of closed subsets of $X$ such that $F_0 =\overline{A}$ and
$\{\int F_\alpha : \alpha < \kappa\}$ is {\em strictly} increasing, moreover $|F_\alpha| = \mathfrak{c}$
for all $\alpha < \kappa$. It is clear that then
$$
    U = \bigcup \{F_\alpha : \alpha < \kappa\} = \bigcup \{\int F_\alpha : \alpha < \kappa\}.
$$
Hence, by our assumption, $U$ is clopen with $A \subset U$ and $|U| = \mathfrak{c}$,
moreover the strictly increasing open cover $\{\int F_\alpha : \alpha < \kappa\}$ of $U$ witnesses
$L(U) \ge \kappa$ because $\kappa$ is regular.

\smallskip

(ii) If $\mathfrak{c}$ is regular then we may just use part (i) for $\kappa = \mathfrak{c}$ to get 
the clopen $U$ with $|U| = L(U) = \mathfrak{c}$  because by our assumption the union of every increasing
$\mathfrak{c}$-sequence of closed subsets of $X$ is closed.

If, however, $\mathfrak{c}$ is singular then we first fix a sequence $\{\kappa_\xi : \xi < cf(\mathfrak{c})\}$
of regular cardinals $\kappa_\xi \ge cf(\mathfrak{\mathfrak{c}})$ that converges to $\mathfrak{c}$. Then we may repeatedly apply part (i) to obtain
an increasing sequence $\{U_\xi : \xi < cf(\mathfrak{c})\}$ of clopen sets in $X$ such that $|U_\xi| = \mathfrak{c}$
and $L(U_\xi) \ge \kappa_\xi$ for all $\xi < cf(\mathfrak{c})\}$. In fact, what we do to get $U_\xi$
given $\{U_\eta : \eta < \xi\}$, is using part (i) with the choice $A = \bigcup \{U_\eta : \eta < \xi\}$.

Then, using again that $t(x, X) < cf(\mathfrak{c})$ for all $x \in X$, the set $U = \bigcup \{U_\xi : \xi < cf(\mathfrak{c})\}$ is
clopen with $|U| = \mathfrak{c}$, moreover we have $L(U) \ge L(U_\xi) \ge \kappa_\xi$  for all $\xi < cf(\mathfrak{c})\}$,
hence $L(U) = \mathfrak{c}$.
\end{proof}

We note that by $|X \setminus U| = |X| > \mathfrak{c}$ and $X$ being locally $\mathfrak{c}$,
we trivially have $L(X \setminus U) = |X| > \mathfrak{c}$.

\begin{corollary}\label{co:pst}
Assume that $X$ is a {\em locally Lindelöf} regular space with $\psi(X) = t(X) = \omega$ and $|X| > \mathfrak{c}$.
Then there is a clopen subset $U$ of $X$ with $|U| = L(U) = \mathfrak{c}$.
\end{corollary}

\begin{proof}
It follows from Shapirovskii's strengthening of \arhang's theorem in \cite{Sap}, see also 2.27 of \cite{J},
that $X$ is locally $\mathfrak{c}$.

Since we have $t(X) = \omega$, the $\mathfrak{c}$-fair property of $X$ follows if we can show that $|\overline{S}| \le \mathfrak{c}$ whenever $S$ is
any countable subset of $X$. But for such a set $S$  we have $L(\overline{S}) \le w(X) \le \mathfrak{c}$ by the regularity of $X$,
and then $|\overline{S}| \le \mathfrak{c}$ follows since we already know that $X$ is locally $\mathfrak{c}$.

Finally, the assumption $t(X) = \omega < cf(\mathfrak{c})$ implies that we may apply part (ii) of Theorem \ref{tm:lcf}
to conclude that there is a clopen subset $U$ of $X$ with $|U| = L(U) = \mathfrak{c}$.
\end{proof}

Now, if $X$ is almost compact and first countable then it is locally compact, $\psi(X) = t(X) = \omega$, and
for every clopen $U \subset X$ either $U$ or $X \setminus U$ is compact, hence it is immediate from Corollary
\ref{co:pst} that $|X| \le \mathfrak{c}$. This shows that \arhang's theorem may be extended to almost compact spaces,
giving the answer to our motivating question.

If, in addition to the assumptions of Corollary \ref{co:pst}, $X$ is also connected then $|X| > \mathfrak{c}$
would lead to a contradiction, hence we get that any locally Lindelöf and connected regular space $X$ with $\psi(X) = t(X) = \omega$
has cardinality $\le \mathfrak{c}$.

We note that the assumption of regularity of $X$ in the last statement, and hence in
Corollary \ref{co:pst} as well, cannot be weakened to
the Hausdorff property. Indeed, by Corollary 2.6 of \cite{JSSz} there is a locally countable anti-Urysohn
space $X$ with $|X| = 2^\mathfrak{c}$, and any anti-Urysohn
space is connected in a strong sense. (We recall that a Hausdorff space is anti-Urysohn if any two non-empty
regular closed sets in it intersect.) But any locally countable space clearly satisfies all the other
assumptions of Corollary \ref{co:pst}.

In contrast to this, the following corollary of Theorem \ref{tm:lcf} for connected spaces requires only the Hausdorff property.

\begin{corollary}\label{co:conseq}
If $X$ is any connected, locally $\mathfrak{c}$ and sequential Hausdorff space, then $|X| \le \mathfrak{c}$.
\end{corollary}

\begin{proof}
Assume, on the contrary, that $|X| > \mathfrak{c}$.
It is well-known that sequential spaces have countable tightness, moreover the closure of any countable set in a sequential Hausdorff space
has cardinality $\le \mathfrak{c}$, and these
clearly imply that $X$ is $\mathfrak{c}$-fair. Moreover, $t(X) = \omega$ also implies that
the union of every increasing $\omega_1$-sequence of closed sets in $X$ is closed. Thus we may apply Theorem \ref{tm:lcf}
to obtain a clopen subset $U$ of $X$ with $|U| = \mathfrak{c}$. But as $X$ is connected, then we would have $X = U$,  contradicting $|X| > \mathfrak{c}$.
\end{proof}

\def\cprime{$'$}
\makeatletter \renewcommand{\@biblabel}[1]{\hfill[#1]}\makeatother

\end{document}